\documentclass[11pt,oneside]{article}
\usepackage{afterpage}
\usepackage{fancyhdr}
\usepackage{lipsum}
\newcommand\shorttitle{Ricci-curvature of homogeneous Finsler spaces}
\newcommand\authors{s. rani, g. shanker}

\usepackage{supertabular, setspace,hyperref}
\usepackage[a4paper, total={6.5in, 9.5in}]{geometry}
\usepackage[utf8]{inputenc}
\usepackage{amsmath}
\usepackage{mathtools}
\usepackage{amsfonts}
\usepackage{amssymb}
\usepackage{amsthm}
\usepackage{caption}
\usepackage{subcaption}
\usepackage{authblk}
\usepackage[sort]{cite}
\newtheorem{theorem}{Theorem}[section]

\newtheorem{example}[theorem]{Example}
\newtheorem{cor}[theorem]{Corollary}

\newtheorem{remark}{\sc Remark}
\newtheorem{lemma}{\sc Lemma}[section]
\newtheorem{corollary}{\sc Corollary}[section]
\newtheorem{definition}{\sc Definition}[section]

\newcommand{\be}{\begin{eqnarray}}
\newcommand{\ee}{\end{eqnarray}}
\newcommand{\Be}{\begin{eqnarray*}}
	\newcommand{\Ee}{\end{eqnarray*}}
\newcommand{\bee}{\begin{equation}}
\newcommand{\eee}{\end{equation}}
\newcommand{\ba}{\begin{array}}
	\newcommand{\ea}{\end{array}}
\newcommand{\bl}{\begin{lemma}}
	\newcommand{\el}{\end{lemma}}
\newcommand{\bd}{\begin{definition}}
	\newcommand{\ed}{\end{definition}}
\newcommand{\bt}{\begin{theorem}}
	\newcommand{\et}{\end{theorem}}
\newcommand{\bp}{\begin{proof}}
	\newcommand{\ep}{\end{proof}}
\newcommand{\bi}{\begin{itemize}}
	\newcommand{\ei}{\end{itemize}}
\newcommand{\br}{\begin{remark}}
	\newcommand{\er}{\end{remark}}
\newcommand{\bc}{\begin{corollary}}
	\newcommand{\ec}{\end{corollary}}
\newcommand{\bex}{\begin{example}}
	\newcommand{\eex}{\end{example}}

\usepackage{chngcntr}

\counterwithin*{equation}{section}
\counterwithin*{equation}{section}
\begin{document}
	\afterpage{\cfoot{\thepage}}
	\clearpage
	\date{}

	\title{\textbf{On the  Ricci curvature of   homogeneous Finsler spaces with  $(\alpha,\beta)$-metrics}}
	
	\author[1]{Sarita Rani }
	\author[1]{Gauree Shanker\thanks{corresponding author}}
	\affil[1]{\footnotesize Department of Mathematics and Statistics,
		Central University of Punjab, Bathinda, Punjab-151 401, India. \newline 
Email: saritas.ss92@gmail.com,  grshnkr2007@gmail.com	}
	\maketitle
	\begin{center}
		\textbf{Abstract}
	\end{center}
	\begin{small}
		The study of curvature properties of homogeneous Finsler spaces with $(\alpha, \beta)$-metrics is one of the central problems in Riemann-Finsler geometry.  In this paper, we consider  homogeneous Finsler spaces with  square metric and Randers change of square metric.  First, we derive the  explicit formulae  for Ricci curvature of  homogeneous Finsler spaces with these metrics. Next, we find a necessary and sufficient condition under which a homogeneous Finsler space with either of these  metrics is of vanishing $S$-curvature. The  formulae for Ricci curvature of  homogeneous Finsler spaces with square  metric and Randers change of square metric having vanishing $S$-curvature are  established. Finally, we prove that the aforesaid spaces having vanishing $S$-curvature and negative Ricci curvature must be Riemannian. 
	\end{small}

	\section{Introduction}
	Finsler geometry is an interesting and active area of research for both pure and applied reasons (for detail, see   \cite{Antonelli,S.I.AmaH.Nag,RGarGWil1996,GKro1934}). 
	One special class of Finsler spaces is homogeneous and symmetric Finsler spaces, which is  an active area of research nowadays. Many authors  (for detail, see   \cite{SDeng2009, SDenXWan2010, ZHuSDen2012, MXuSDen2015, SRGS2019, 	GSSR2019,LH2015}) have worked in this area. Ricci curvature, denoted by $ Ric(x,y)$  is an important entity in Riemann-Finsler geometry.
	It is the mathematical object that controls growth rate of volume of a metric ball in a manifold.
	The Ricci curvature of a Finsler space $(M,F)$ can always be expressed in the form
	$ Ric(x,y)=(n-1) \lambda(x,y) F^2, $ where $ \lambda(x,y) $ is a scalar function on $TM,$ called {Einstein scalar.}
	If the Einstein scalar $\lambda(x,y)$ depends only on $x,$ then the Finsler metric $F$ is called an {Einstein metric}. Einstein manifolds play an important role in Riemann-Finsler geometry. Motivated by the open problem 
	\begin{center}
		 ``Does every smooth manifold admits an Einstein Finsler metric?"
	\end{center}
	posed  by Chern, so many  geometers  (for example, see    \cite{BRS2004, AB1987, CST2012, SD2015, RRR2012, ZS2001, YY2010}) have established many results on Einstein Finsler metrics. \\
	 In  \cite{BRARNS2008},  Rezaei et al.  studied some Einstein $ (\alpha, \beta) $-metrics. Further, in \cite{HWanSDen2010}, Wang and Shen  have  studied Einstein-Randers metrics on  homogeneous Riemannian spaces. After that,  Deng and  Hu \cite{SDengZHu2013} extended the work done in \cite{HWanSDen2010} to homogeneous Finsler spaces.\\
	In \cite{DBaoCRob2004},  Bao and  Robles  proved that a compact Einstein-Randers space  without boundary and with negative Ricci curvature must be Riemannian. The problem of negative Ricci curvature was further studied by Mo and Yu in  \cite{XCT2008}.\\
	In \cite{LZho2010}, Zhou  studied Finsler spaces with  $ (\alpha, \beta) $-metrics and introduced  formulae for Riemann curvature and Ricci curvature for these metrics. Further,  Cheng et al. \cite{CST2012} studied  Einstein  $ (\alpha, \beta) $-metrics. During their study, they  found that the formulae given in \cite{LZho2010} are incorrect and they provided correct version of these formulae. Based on these formulae, a formula for Ricci curvature of homogeneous  $ (\alpha, \beta) $-metrics was constructed  by  Yan and Deng \cite{ZYSD2016}.\\
	 In 1972,  Matsumoto \cite{M.Mat1972}  introduced the concept of $(\alpha, \beta)-$metrics which are the generalizations of Randers metric introduced by  Randers \cite{Randers}. The main aim of this paper is to establish an explicit formula for Ricci-curvature  of a  homogeneous Finsler space with  Shen's square metric, and the same for that metric having vanishing $S$-curvature. \\	
	The simplest non-Riemannian metrics are the Randers metrics given by $F=\alpha+ \beta$ with $\lVert \beta\rVert_\alpha <1, $ where $\alpha$ is a Riemannian metric and $\beta$ is a 1-form. Besides Randers metrics, other interesting kind of non-Riemannian metrics are square metrics. Berwald's metric, constructed by Berwald \cite{LBer1929} in 1929 as 
	$$ F=\dfrac{\left( \sqrt{\left( 1-\lvert x\rvert^2\right)\lvert y\rvert^2+\langle x,y \rangle^2 }+\langle x,y \rangle \right)^2 }{\left( 1-\lvert x\rvert^2\right)^2 \sqrt{\left( 1-\lvert x\rvert^2\right)\lvert y\rvert^2+\langle x,y \rangle^2 }} $$ 
	is a classical example of square metric. Berwald's metric can be rewritten as follows:
	\begin{equation}{\label{s1}}
	F=\dfrac{\left( \alpha+ \beta\right) ^2}{\alpha},
	\end{equation}
	where $$ \alpha= \dfrac{\sqrt{\left( 1-\lvert x\rvert^2\right)\lvert y\rvert^2+\langle x,y \rangle^2 }}{\left( 1-\lvert x\rvert^2\right)^2}, $$
	and 
	$$\beta=\dfrac{\langle x, y \rangle}{\left( 1-\lvert x\rvert^2\right)^2}. $$
	Equation (\ref{s1}) with the condition  $\lVert \beta\rVert_\alpha <1 $ satisfies Lemma \ref{Shenlemma} which is a  necessary and sufficient condition for $F$ 
	to be a  Finsler  metric. Therefore, the  metric given by equation  (\ref{s1}) is an $ (\alpha, \beta)$-metric.
	An $(\alpha, \beta)$-metric expressed in the form (\ref{s1}) is called  square metric \cite{ZSheCYu2014}. Square metrics play an important role in Finsler geometry. The importance of square metric can be seen in papers \cite{ZSheCYu2014, J2QXia2016, CYuHZhu2015}. Square metrics can also be expressed in the form 
	$$F=\dfrac{\left( \sqrt{\left( 1-b^2\right)\alpha^2+\beta^2 }+\beta \right)^2 }{\left( 1-b^2\right)^2 \sqrt{\left( 1-b^2\right)\alpha^2+\beta^2 }}, $$
	where $ b:= \lVert \beta_x\rVert_\alpha $ is the length of $\beta$ \cite{CYuHZhu2015}. \\
	In this case, $F=\alpha \phi\left( b^2, \dfrac{\beta}{\alpha}\right) ,$ where $\phi = \phi (b^2,s) $ is a smooth function, is called general $(\alpha, \beta)$-metric.
	If $\phi= \phi(s)$ is independent of $b^2,$ then $F$ is called an $(\alpha, \beta)$-metric.\\
	\noindent	If  $(M, F)$ be an $n$-dimensional Finsler space and $\beta=b_i(x) y^i$ is a $1$-form on $M.$ Then $ F \longrightarrow \bar{F} $ is called Randers change if 
		\begin{equation}{\label{s2}}
		\bar{F}=F +\beta 
		\end{equation}
	Above  change 
	 has been introduced by  Matsumoto \cite{M.Mat1974}, and it was named as ``Randers change" by  Hashiguchi and  Ichijy$\bar{o}$ \cite{MHasYIch1980}. 
	In the current paper, we deal with	square metrics 
		 $$F= \dfrac{(\alpha+\beta)^2}{\alpha}= \alpha \phi(s), \text{where} \ \ \phi(s)=1+s^2+2s.$$
	\noindent The paper is organized as follows:\\
	In section 2, we discuss some basic definitions and results to be used in consequent sections.  The explicit formulae  for Ricci curvature of  homogeneous Finsler spaces with square metric and Randers change of square metric  have been established in section 3. In section 4, we find a necessary and sufficient condition under which a homogeneous Finsler space with either of these metrics  is of vanishing $S$-curvature. Also, a formula for Ricci curvature of a homogeneous Finsler space with each  metric having vanishing $S$-curvature is established. Finally, we prove that the aforesaid spaces having vanishing $S$-curvature and negative Ricci curvature must be Riemannian. 
	\section{Preliminaries}
	In this section, we discuss some basic definitions and results required for consequent sections.  We refer \cite{BCS, ChernShenRFG, SDeng2012} for notations and further details.
	\begin{definition}
		An n-dimensional real vector space $V$ is said to be  a \textbf{Minkowski space}
		if there exists a real valued function $F:V \longrightarrow [0,\infty),$ called Minkowski norm, satisfying the following conditions: 
		\begin{itemize}
			\item  $F$ is smooth on $V \backslash \{0\},$ 
			\item $F$ is positively homogeneous, i.e., $ F(\lambda v)= \lambda F(v), \ \ \forall \ \lambda > 0, $
			\item For any basis $\{u_1,\ u_2, \,..., \ u_n\}$ of $V$ and $y= y^iu_i \in V,$ the Hessian matrix \\
			$\left( g_{_{ij}}\right)= \left( \dfrac{1}{2} F^2_{y^i y^j} \right)  $ is positive-definite at every point of $V \backslash \{0\}.$
		\end{itemize}
	\end{definition}
	\begin{definition}	
		Let $M$ be a connected   smooth manifold. If there exists a function $F: TM \longrightarrow [0, \infty)$ which  is smooth on slit tangent bundle $TM \backslash \{0\}$ and the restriction of $F$ to any $T_xM, \ x \in M,$ is a Minkowski norm, then $M$ is called a Finsler space  and  $F$ is called a Finsler metric. 
	\end{definition}
	An $(\alpha, \beta)$-metric on a connected smooth manifold $M$ is a Finsler metric $F$ constructed from a Riemannian metric $ \alpha = \sqrt{a_{ij}(x) y^i y^j} $ and a one-form $ \beta = b_i(x) y^i $ on $M$ and is of the form  $ F= \alpha \phi \left( \dfrac{\beta}{\alpha}\right),$ where $ \phi  $ is a smooth function on $M.$ Basically,  $(\alpha, \beta)$-metrics are the generalization of Randers metrics.   Many authors  (for example, see   \cite{SDenXWan2010, ZHuSDen2012, GK2019,GSK2019, MXuSDen2015}) have worked on $(\alpha, \beta)$-metrics.	 Let us recall following (Shen's) lemma	   which gives necessary and sufficient condition for a function of  $\alpha$ and  $\beta$  to be a Finsler metric: 
	\begin{lemma} {\label{Shenlemma}} \cite{ChernShenRFG}
		Let $F=\alpha \phi(s), \ s=\beta/ \alpha,$ where $\phi $ is a smooth function on an open interval  $ (-b_0, b_0), \ \alpha$ is a Riemannian metric and $\beta$ is a 1-form with  $\lVert \beta \rVert_{\alpha} < b_0.$ Then $F$ is a Finsler metric if and only if the following conditions are satisfied:
		$$ \phi(s) > 0, \ \  \phi(s)-s\phi'(s)+\left( b^2-s^2\right) \phi''(s)>0, \ \  \forall \ \ \lvert s\rvert \leq b < b_0.$$
	\end{lemma}
\subsection{Ricci curvature of an $(\alpha, \beta)$-metric}
The notion of Riemannian curvature for Riemannian spaces can be extended to Finsler spaces.  For an $n$-dimensional Finsler space $(M,F),\ x \in M, \ y (\neq 0) \in T_xM,$ the Riemannian curvature is a linear map $ R_y : T_xM \longrightarrow T_xM $ defined as 
$$ R_{y}(v)= R^i_k(y)v^k \dfrac{\partial}{\partial x^i}, \ \ v= v^i \dfrac{\partial}{\partial x^i},$$
where
$$R^i_k(y)= 2 \dfrac{\partial G^i}{\partial x^k}- \dfrac{\partial^2G^i}{\partial x^j \partial y^k}y^j + 2 G^j \dfrac{\partial^2 G^i}{\partial y^j \partial y^k}- \dfrac{\partial G^i}{\partial y^j}\dfrac{\partial G^j}{\partial y^k}, $$
and $ G^i$  are geodesic coefficients  given by 
$$ G^i \ := \dfrac{1}{4} g^{im} \left\{ \left( F^2\right)_{x^k y^m} y^k - \left( F^2\right)_{x^m}  \right\}, \ \ i=1,2, ...,n. $$

\begin{definition}
	Let $(M,F)$ be an n-dimensional Finsler space. For $ x \in M $ and non-zero $ y \in T_xM, $ the map defined on tangent bundle $TM$ of $M$ as 
	$$ Ric(y)= tr(R_{y}) \ \forall \ y \in T_xM$$
	is called \textbf{Ricci-curvature} of $(M,F).$
\end{definition}
%
%
For an $(\alpha, \beta)$-metric $F=\alpha \phi (s),\ s =\dfrac{\beta}{\alpha}$ on  a Finsler manifold $M,$  let us take
\begin{equation*}
\begin{aligned}
r_{ij}&= \dfrac{b_{i;j}+ b_{j;i} }{2},\ \ \ \ \ \ \ \ \ \ \ \ \ \ \ \ \ \ \ \ \ \ \ \ \ \ \ \ \ \ \ \ \ \ \ \   && s_{ij}= \dfrac{b_{i;j} - b_{j;i} }{2}, \ \ \ \ \ \ \ \ \ \ \ \ \ \ \ \ \ \ \ \ \ \ \ \ \ \ \  \\
r^i_j&= a^{i{\ell}} r_{{\ell}j}, \ \ && s^i_j= a^{i{\ell}} s_{{\ell}j},  \\ 
r_i&=b^{\ell}r_{{\ell}i}=b_j r^j_i, \ \ && s_i=b^{\ell}s_{{\ell}i}=b_j s^j_i,\\
r&= r_{ij}b^ib^j= b^ir_i, \ \ && r_{00}=r_{ij}y^iy^j, \\ 
r_{i0}&=r_{ij}y^j,\ \ &&s_{i0}=s_{ij}y^j,\\ 
r_0&=r_i y^i,\ \ && s_0=s_iy^i,
\end{aligned}
\end{equation*}
where $ ; $ denotes the covariant derivative w.r.t. Levi-Civita connection on Riemannian metric $\alpha$ and $ a^{ij}= \left( a_{ij}\right) ^{-1}, \ \ b^i=a^{ij}b_j. $
 \begin{theorem}\label{riccicur}   \cite{CST2012}
 	Let  F be an $(\alpha,\beta)$-metric on a Finsler space $M$ and $^\alpha Ric $ be Ricci curvature of $\alpha.$
 	Then  Ricci curvature  of $F$ is given by $ Ric= ^\alpha Ric + RT^{l}_{l},$ with
 	\begin{align*}
 	RT^{l}_{l}=& \frac{1} {\alpha^2}\Big \{\left(n-1 \right)\zeta_1 + \zeta_2\Big \} r^2_{00} \\
 	& + \dfrac{1}{\alpha} \Big [\Big( \left(n-1 \right) \zeta_3+\zeta_4 \Big) r_{00}s_{0} + \Big( \left(n-1 \right) \zeta_5+\zeta_6 \Big) r_{00}r_{0}+ \Big( \left(n-1 \right) \zeta_7+\zeta_8 \Big) r_{00;0}\Big] \\
 	&+\Big(  \left(n-1 \right)\zeta_9 + \zeta_{10}\Big) s^2_0 + \left(rr_{00}-r^2_0 \right)\zeta_{11} + \Big(\left(n-1 \right)\zeta_{12} + \zeta_{13}\Big) r_0 s_0 \\
 	& + \Big( r_{00}r^{l}_{l}-r_{0{l}}r^{l}_0+r_{00;{l}}b^{l}-r_{0{l};0}b^{l}\Big)\zeta_{14} +\Big(\left(n-1 \right)\zeta_{15} + \zeta_{16}\Big) r_{0{l}}s^{l}_0 \\
 	& +\Big(\left(n-1 \right)\zeta_{17} + \zeta_{18}\Big) s_{0;0} +  s_{0{l}}s^{l}_0 \zeta_{19}+ \alpha \Big[ rs_0\ \zeta_{20}+ \Big(\left(n-1 \right)\zeta_{21} + \zeta_{22}\Big) s_{l} s^{l}_0 \Big] \\
 	&+ \alpha \Big( ( 3s_{l}r^{l}_0-2s_0r^{l}_{l}+2r_{l}s^{l}_0-2s_{0;{l}}b^{l}+s_{{l};0}b^{l})\ \zeta_{23} +s^{l}_{0;{l}}  \zeta_{24} \Big) \\
 	&+ \alpha^2 \left( s_{l}s^{l}\ \zeta_{25}+s^i_{l}s^{l}_i\ \zeta_{26}\right), 
 	\end{align*}
 	where
 \begin{align*}
\zeta_1=&2\,\psi\,\Theta_{{s}} \left( B-{s}^{2} \right) -2\,s\psi\,\Theta+{
	\Theta}^{2}-\Theta_{{s}},\\
\zeta_2=&2\,\psi\,\psi_{{{\it ss}}} \left( B-{s}^{2} \right)^{2} - \left( 6\,s\psi
\,\psi_{{s}}+\psi_{{{\it ss}}} \right)  \left( B-{s}^{2} \right) +2\,s
\psi_{{s}},\\
\zeta_3=&  -4\, \left( 2\,Q\Theta_{{s}}+Q_{{s}}\Theta \right) \psi\, \left( B-{s}
^{2} \right) +4\,Q\Theta_{{s}}+2\,Q_{{s}}\Theta+4\,Q\Theta\, \left( s
\psi-\Theta \right) -2\,\Theta_{{B}}, \\
\zeta_4=& -4\,\psi\, \left( 2\,Q\psi_{ss}+Q_s\psi_{{s}}+Q_{ss}{
	\psi_{{s}}}^{2} \right)  \left( B-{s}^{2} \right) ^{2}\\&+ \left( -4\,{
	\psi}^{2} \left( Q-sQ_{{s}} \right) +4\,Q_{{{\it ss}}}\psi+2\,Q_{{s}}
\psi_{{s}}+4\,Q\psi_{{{\it ss}}}-2\,\psi_{{{\it sB}}}+20\,sQ\psi\,\psi
_{{s}} \right)  \left( B-{s}^{2} \right)\\& +2\,\psi\, \left( Q-sQ_{{s}}
\right) -4\,\psi_{{s}}-Q_{{{\it ss}}}-10\,sQ\psi_{{s}}, \\
\zeta_5=& 4\,\psi\,\Theta-2\,\Theta_{{B}}, \\
\zeta_6=& 2\, \left( 2\,\psi\,\psi_{{s}}-\psi_{{{\it sB}}} \right)  \left( B-{s}
^{2} \right) -2\,\psi_{{s}}, \\
\zeta_7=& -\Theta, \\
	\end{align*}
\begin{align*}
\zeta_8=& -\psi_s\ (B-s^2),  \\
\zeta_9=& 8\,Q\psi\, \left( Q\Theta_{{s}}+Q_{{s}}\Theta \right)  \left( B-{s}^{2
} \right) +4\,{Q}^{2} \left( {\Theta}^{2}-\Theta_{{s}} \right) +4\,Q
\left( \Theta_{{B}}-Q_{{s}} \right),  \\
\zeta_{10}=& \left( 4\,{\psi}^{2} \left( 2\,QQ_{{{\it ss}}}-{Q_{{s}}}^{2} \right) 
+8\,Q\psi\, \left( Q\psi_{{{\it ss}}}+Q_{{s}}\psi_{{s}} \right) -4\,{Q
}^{2}{\psi_{{s}}}^{2} \right)  \left( B-{s}^{2} \right) ^{2}\\&+ \left( -
16\,sQ\psi\, \left( Q\psi_{{s}}+Q_{{s}}\psi \right) -4\,\psi\, \left( 
2\,QQ_{{{\it ss}}}-{Q_{{s}}}^{2} \right) -4\,Q \left( Q\psi_{{ss}}+Q_{{s}}\psi_{{s}} \right) +4\,Q\psi_{{{ sB}}}+4\,Q_{{s}}\psi_{{B
}} \right)  \left( B-{s}^{2} \right)\\& -4\,{s}^{2}{Q}^{2}{\psi}^{2}+4\,
\left( 2+3\,sQ \right)  \left( Q\psi_{{s}}+Q_{{s}}\psi \right) -8\,{Q
}^{2}\psi+2\,QQ_{{{\it ss}}}-{Q_{{s}}}^{2}+4\,sQ\psi_{{B}}, \\
\zeta_{11}=& 4\,\psi^2+4\,\psi_B,  \\
\zeta_{12}=& 4\,Q\,(-2\, \psi \Theta + \Theta_{{B}}), \\
\zeta_{13}=& \Big( 8\,\psi\, \left( Q_{{s}}\psi-Q\psi_{{s}} \right) +4\,Q\psi_{{ sB}}+4\,Q_{{s}}\psi_{{B}} \Big)  \left( B-{s}^{2} \right) +8\,sQ{\psi}^{2}+4\,Q\psi_{{s}}-4\, \left( 1-sQ \right) \psi_{{B}}, \\
\zeta_{14}=& 2\, \psi,  \\
\zeta_{15}=& 4\,Q\, \Theta, \\
\zeta_{16}=& 4\, \left( Q\psi_{{s}}-Q_{{s}}\psi \right)  \left( B-{s}^{2} \right) + 2\,Q_{{s}}-2\, \left( 1+2\,sQ \right) \psi, \\
\zeta_{17}=& 2\,Q\, \Theta, \\
\zeta_{18}=& 2\, \left( Q\psi_{{s}}+Q_{{s}}\psi \right)  \left( B-{s}^{2} \right) + 2\,sQ\psi-Q_{{s}}, \\
\zeta_{19}=& 2\, \left( 1+sQ \right) Q_{{s}}-2\,{Q}^{2}, \\  
\zeta_{20}=& -8\, \left( {\psi}^{2}+\psi_{{B}} \right) Q, \\
\zeta_{21}=& -4\,{Q}^{2}\Theta,  \\
\zeta_{22}=&  2\,Q\psi-4\,{Q}^{2}\psi_{{s}} \left( B-{s}^{2} \right),  \\
\zeta_{23}=& 2\, Q\, \psi,  \\
\zeta_{24}=& 2\, Q,  \\
\zeta_{25}=&  -4\,{Q}^{2}\psi,  \\
\zeta_{26}=&   -Q^2, \\
Q=&\dfrac{\phi'}{\phi-s\phi'},\\
\Theta=&\dfrac{\phi \phi'-s(\phi\phi''+\phi'\phi')}{2\phi\left(\phi-s\phi' +(B-s^2)\phi'' \right) },\\
\psi=& \dfrac{\phi''}{2\phi\left(\phi-s\phi' +(B-s^2)\phi'' \right)},\\
B=&b^2.
\end{align*}
\end{theorem} 
\subsection{Homogeneous Finsler spaces}

Let $G$ be a Lie group and $M,$ a smooth manifold. Then a smooth  map $ f: G \times M \longrightarrow M  $ satisfying
$$ f(g_2, f(g_1,x)) =f(g_2g_1, x), \ \ \text{for all }\ \ g_1, g_2 \in G, \ \ \text{and }\ \ x \in M$$
is called a smooth action of $G $ on $M.$


\begin{definition}
	Let $ M $ be a smooth manifold and  $ G, $ a Lie group. If $ G $ acts smoothly  on $ M, $ then $ G $ is called a \textbf{Lie transformation group} of $ M. $
\end{definition}
The following theorem  gives us a differentiable structure on the coset space of a Lie group.
\begin{theorem}{\label{DiffStrct.}}
	Let $G$ be a Lie group and $H,$ its  closed subgroup. Then there exists a unique differentiable structure on the left coset space $G/H$ with the induced topology that turns $G/H$ into a smooth manifold such that $G$ is a Lie transformation group of $G/H.$
\end{theorem}
\begin{definition}
	Let $ (M, F)$  be a  connected Finsler space and  $ I(M, F)$ the group of isometries of $(M,F).$ If the action of $I(M,F)$ is transitive on $M,$ then $(M,F)$  is said to be a \textbf{homogeneous Finsler space}. 
\end{definition}


Let $G$ be a Lie group acting transitively on  a smooth manifold $M.$ Then for $a \in M,$ the isotropy subgroup $G_a$ of $G$ is a closed subgroup and by Theorem  \ref{DiffStrct.}, $G$ is a Lie transformation group of $ G/G_a.$ Further,  $ G/G_a $ is diffeomorphic to $M.$ 

\begin{theorem}\cite{SDeng2012}
	Let $(M,F)$ be a Finsler space. Then $G=I(M,F),$ the group of isometries  of $M$ is a Lie transformation group of $M.$ Let $ a \in M$ and $ I_a(M,F)$ be the isotropy subgroup of $I(M,F)$ at $a.$ Then $ I_a(M,F)$ is compact.
\end{theorem}
Let $ (M,F)$ be a  homogeneous Finsler space, i.e., $G= I(M,F)$ acts transitively on $M.$ For $a \in M,$ let $ H=I_a(M,F) $ be a closed isotropy subgroup of $G$ which is compact. Then $H$ is a Lie group itself being a closed subgroup of $G.$ Write $M$ as the quotient space $G/H.$
\begin{definition}\cite{Nomizu}{\label{defNomizu}}
	Let $ \mathfrak g$ and $ \mathfrak h $ be the Lie algebras of the Lie groups $G$ and $H$ respectively. Then the direct sum decomposition of $\mathfrak{g} $ as  $ \mathfrak g = \mathfrak h + \mathfrak k,$ where $\mathfrak k$ is a subspace of $\mathfrak g$   such that $\text{Ad}(h) (\mathfrak k) \subset \mathfrak k \ \ \forall \  h \in H,$ is called a reductive decomposition of $\mathfrak g,$ and if such decomposition exists, then $ (G/H, F) $ is called reductive homogeneous space.
\end{definition}
Therefore, we can write, any homogeneous Finsler space as a coset space of a connected Lie group with an invariant Finsler metric. Here, the Finsler metric $F$ is viewed as $G$ invariant Finsler metric on $M.$\\

 We can  identify $\mathfrak{k}$  with the tangent space $T_{H}\left( G/H\right)  $ of $G/H$  at the origin $H.$    Let $\left\lbrace  w_1, w_2, w_3, ... , w_n  =\dfrac{w}{c} \right\rbrace $ be the orthonormal basis of $\mathfrak{k},$
where $w$ is the invariant vector field of length $c$ corresponding to the $1$-form $\beta,$ i.e., $ \langle w, Y\rangle =\beta (Y),\  \forall \ Y \in  \mathfrak{k},$ and  $ \langle \ , \ \rangle$ is the restriction of Riemannian metric $\alpha$ to $\mathfrak{k}.$
\section{Ricci Curvature  of Homogeneous Finsler Spaces with  Certain $(\alpha, \beta)$-Metrics}
$\Gamma^m_{ij}$ are the  Christoffel symbols   given by 
$$\nabla_{\dfrac{\partial}{\partial x^i}} \dfrac{\partial}{\partial x^j}=\Gamma^m_{ij}\ \dfrac{\partial}{\partial x^m}.$$
To compute $\Gamma^k_{ij},$ we need following notations:\\
For $1\leq i, \ j, \ m \leq n,$ define the structure constants of $\mathfrak k$ as follows:  
$$C^m_{ij}=\left\langle \left[ w_i, w_j\right]_{\mathfrak k}, w_m  \right\rangle, $$
where $\left[ w_i, w_j\right]_{\mathfrak k}$ denotes the projection of $\left[ w_i, w_j\right]$ to $\mathfrak k.$ \\
Take $$\sum_{t=1}^{n}C^t_{ij}C^t_{ml}=C^t_{ij}C^t_{ml},$$ and   
$$\left\langle \left[ w_i, w_j\right]_{\mathfrak k}, Y  \right\rangle=C^0_{ij}$$ for any non-zero $ Y \in \mathfrak k. $\\
Define 
$$ f\left( i, j\right)= \begin{cases}
1, & \text{if}\ \  i<j, \\ 
0, & \text{if}\ \  i \geq j.
\end{cases}  $$
With above notations, let us recall   following lemmas for later use:
\begin{lemma}\cite{SDengZHu2013}  \label{lemma3.1}
	At the origin $  o  =eH,$	we have:
	\begin{enumerate}
		\item  [(i)]  $ 	\Gamma^m_{ij}=f(i,j)C^m_{ij} + \left\langle \nabla_{\hat{w}_i} \hat{w}_j, \hat{w}_m\right\rangle,$
		\item [(ii)] $	\left\langle \nabla_{\hat{w}_i} \hat{w}_j, \hat{w}_m\right\rangle = -\dfrac{1}{2}\left( C^i_{jm}+C^j_{im}+C^m_{ij}\right),$ 
		\item[(iii)] $	\left\langle \nabla_{\hat{w}_i} \hat{w}_j, \hat{w}_m\right\rangle \hat{w}_l = \dfrac{1}{2}\left(C^i_{lt}C^t_{jm}+ C^j_{lt}C^t_{im}+ C^m_{lt}C^t_{ij} +  C^t_{jm}C^t_{li} + C^t_{im}C^t_{lj}+ C^t_{ij}C^t_{lm}\right),$\\
		where $w_i, w_j, w_m, w_l \in \mathfrak k.$
	\end{enumerate}
\end{lemma}
\begin{lemma}\label{lemma3.2}\cite{SDengZHu2013}  At the origin $ o=eH,$ we have: 
	\begin{align*}
	b_k&=c\delta_{nk},\\
	s_{jk}&= \dfrac{c}{2}C^n_{jk},\\
	s_k&= \dfrac{c^2}{2}C^n_{nk},\\
	r_{jk}&=\dfrac{c}{2}\left(C^k_{jn} + C^j_{kn}\right),\\
	s_{jk;m}&=\dfrac{c}{2}\left\{- C^t_{jk}C^n_{mt} + \dfrac{1}{2}C^n_{jt}\left( -C^t_{km}+C^m_{kt}+ C^k_{mt}\right) + \dfrac{1}{2}C^n_{tk}\left(-C^t_{jm}+C^m_{jt}+ C^j_{mt} \right) \right\},\\	s_{j;k}&=c\left(s_{nj;k}+\dfrac{c}{2}C^n_{tj}\Gamma^t_{nk}\right),\\
	b_{j;k;m}&=c\Big( - \Gamma^j_{nt}\left\langle \nabla_{\hat{w}_m} \hat{w}_k, \hat{w}_t \right\rangle - \Gamma^t_{nk} \left\langle \nabla_{\hat{w}_m} \hat{w}_j, \hat{w}_t \right\rangle  + C^t_{mn} \left\langle \nabla_{\hat{w}_t} \hat{w}_k, \hat{w}_j \right\rangle\\&\ \  +  \hat{w}_m \left\langle \nabla_{\hat{w}_n} \hat{w}_k, \hat{w}_j \right\rangle \Big),\\
	r_{jk;m}&=s_{jk;m}+b_{k;j;m}.
	\end{align*}
\end{lemma}
Next, we calculate the quantities used in Theorem \ref{riccicur} at the origin for a homogeneous Finsler space using Lemma \ref{lemma3.1} and \ref{lemma3.2}.
\begin{align*}
r_{00}&= cC^0_{0n},\\
s_j&=b^{l}s_{{l}j} =cs_{nj},\\
s_0&=cs_{n0} =\dfrac{c^2}{2}C^n_{n0},\\
r_j&=b^{l}r_{{l}j} =cr_{nj},\\
r_0&=cr_{n0}=-\dfrac{c^2}{2}\left(C^n_{n0}+C^0_{nn}\right)
=\dfrac{c^2}{2}C^n_{0n},\\
r_{00;0}&=cC^0_{0i}\left( C^0_{in}+C^i_{0n}\right),\\
r&=r_{ij}b^ib^j
=c^2r_{nn}
= 0,\\
r^{l}_{l}&=a^{{l}i}r_{i{l}}=-\dfrac{c}{2}\delta^{{l}i}\left( C^i_{n{l}}+C^{l}_{ni}\right) 
= -\dfrac{c}{2}\left( C^{l}_{n{l}}+C^{l}_{n{l}}\right) 
=cC^{l}_{{l}n},\\
r^{l}_0&=a^{l i}r_{i 0}=\delta^{l i}r_{i 0}=r_{{l}0},
\\
r_{0{l}}r^{l}_0&=r_{0{l}}r_{{l}0}= \left( -\dfrac{c}{2}\right) ^2\left(C^0_{n{l}}+C^{l}_{n0} \right)  \left(C^{l}_{n0}+C^0_{n{l}} \right) =\dfrac{c^2}{4}\left( C^0_{n{l}}+C^{l}_{n0}\right)^2,\\
r_{00;{l}}b^{l}&=r_{00;n}b^n=cr_{00;n}
=c\left( s_{00;n}+b_{0;0;n}\right) 
=\dfrac{c^2}{2}\left( C^t_{0n}+C^0_{tn}\right)\left(C^n_{0t}+C^0_{nt}+ C^t_{0n}\right),
\end{align*}
\begin{align*}
r_{0{l};0}b^{l}&=r_{0n;0}b^n=cr_{0n;0}
=c\left(s_{0n;0}+b_{n;0;0}\right)\\& 
=\dfrac{c^2}{4}\left\{2\ C^n_{nt}C^0_{t0}+ \left(C^n_{t0}+C^0_{tn}+C^t_{0n}\right)\left(C^0_{nt}+C^t_{n0}\right)\right\}, 
\\
s^{l}_0&=a^{l i}s_{i 0}=\delta^{l i}s_{i 0}=s_{{l}0},
\\
r_{0{l}}s^{l}_0&=r_{0{l}}s_{{l}0}
=\dfrac{c^2}{4} \left( C^0_{n{l}}C^n_{0{l}}+C^{l}_{n0}C^n_{0{l}}\right),\\
s_{0;0}&= c s_{n0;0}+\dfrac{c^2}{2}C^n_{{l}0}\Gamma^{l}_{n0} = \dfrac{c^2}{2} C^n_{nt} C^0_{0t},\\
s_{0{l}}s^{l}_0&=s_{0{l}}s_{{l}0}
=-\dfrac{c^2}{4}\left( C^n_{0{l}}\right)^2,\\
s_{l}s^{l}_0&=s_{l}s_{{l}0}
=\dfrac{c^3}{4}C^n_{n{l}}C^n_{{l}0},\\
s_{l}r^{l}_0&
=s_{l}r_{l0}
=-\dfrac{c^3}{4}C^n_{n{l}} \left(C^0_{n{l}}+C^{l}_{n0}\right),
\\
r_{l}s^{l}_0&=r_{{l}}s_{{l}0}
= \dfrac{c^3}{4}C^n_{{l}n}C^n_{{l}0} ,
\\
s_{0;{l}}b^{l}&
=cs_{0;n}
=c\left( c s_{n0;n}+ \dfrac{c^2}{2}C^n_{t0}\Gamma^t_{nn}\right)
=\dfrac{c^3}{4}C^n_{nt}\left(C^t_{0n}+C^0_{nt}+C^n_{0t} \right),
\\
s_{{l};0}b^{l}&
=cs_{n;0}
=c\left( c s_{nn;0}+ \dfrac{c^2}{2}C^n_{tn}\Gamma^t_{n0}\right)
=-\dfrac{c^3}{4}C^n_{tn}\left(C^n_{0t}+C^0_{nt}+ C^t_{n0}\right),
\\
s^{l}_{0;{l}}&
=a^{l i}s_{i0;{l}}
=\delta^{l i}s_{i0;{l}}
=s_{{l}0;{l}}
=\dfrac{c}{4}\left\{2 C^n_{t0}C^{l}_{{l}t}+C^n_{{l}t}\left(C^t_{0{l}}+C^{l}_{0t}+C^0_{{l}t} \right) \right\},\\
s_{l}s^{l}&
=s_{l} s_{l}
=\dfrac{c^4}{4}\left(C^n_{n{l}} \right)^2,\\
s^j_{l} s^{l}_j&
=-\dfrac{c^2}{4}\left( C^n_{j{l}}\right)^2.
\end{align*}
Now, we calculate $\zeta_1$ to $\zeta_{26}$ for square  metric as follows:
\begin{align*}
\zeta_1=&{-\dfrac {3\Big(10\,{s}^{4}-8\,{s}^{3}- \left( 4\,B+3 \right) {s}^{2}+
		\left( 4+4\,B \right) s-1-2\,B\Big)} {\left( 1-3\,{s}^{2}+2\,B \right)^{3}}},\\
\zeta_2=&	{-\dfrac {6\Big(15\,{s}^{6}+ \left( -25\,B-2 \right) {s}^{4}+ \left( 4\,B+
		8\,{B}^{2}-3 \right) {s}^{2}+2\,{B}^{2}+B\Big)}{ \left( 1-3\,{s}^{2}+2\,B
		\right) ^{4}}},\\
\zeta_3=&  {-\dfrac {8\Big(-24\,{s}^{4}+15\,{s}^{3}+ \left( 10+8\,B \right) {s}^{2}+
		\left( -9-6\,B \right) s+4\,B+2\Big)}{ \left( 1-s \right)  \left( 1-3\,{s
		}^{2}+2\,B \right) ^{3}}},  \\
\zeta_4=&    \dfrac{\splitdfrac   {-48\ \left(B- {s}^{2}\right)^2   \bigl( -63\,{s}^{6}+  117{s}^{5}+ \left(36\,B -36 \right) {s}^{4}- \left(30+60\,B\right) {s}^{3}}  {+ \left(4\,{B}^{2}+ 28\,B+25 \right) {s}^{2}- \left(12\,B+3+12\,{B}^{2} \right) s+2+8\,B+8\,{B}^{2} \bigr) }   }  { \left(1-s\right)^{3} \left(1-3\,{s}^{2}+2\,B\right)^{5}}\\&+
{\dfrac { 8\Big( 111\,{s}^{4}-168\,{s}^{3}+ \left( 58-22\,B\right) {s}^{2}-12\,Bs+7+18\,B+8\,{B}^{2} \Big)  \left( B-{s}^{2}\right) }{ \left( 1-3\,{s}^{2}+2\,B \right) ^{3} \left( 1-s \right) ^{3}}}\\&-
{\dfrac {4\Big(39\,{s}^{4}-51\,{s}^{3}+ \left( 7-16\,B \right) {s}^{2}+
		\left( 9+6\,B \right) s+2\,B+4\,{B}^{2}\Big)}{ \left( 1-3\,{s}^{2}+2\,B\right) ^{2} \left( 1-s \right) ^{3}}},
\end{align*}
\begin{align*}
\zeta_5=& {\dfrac {8\Big(1-2\,s\Big)}{ \left( 1-3\,{s}^{2}+2\,B \right) ^{2}}}, \\
\zeta_6=& {\dfrac {12\ s \Big( 4\,B-3\,{s}^{2}-1 \Big) }{ \left( 1-3\,{s}^{2}+
		2\,B \right) ^{3}}},\\
\zeta_7=& {\dfrac {-1+2\,s}{ 1-3\,{s}^{2}+2\,B}}, \\
\zeta_8=& {\dfrac {-6\,s\Big(B-s^2\Big)}{ \left( 1-3\,{s}^{2}+2\,B \right) ^{2}}},\\
\zeta_9=& \dfrac{\splitdfrac  {-16\bigl(-27\,{s}^{6}+ \left( 54\,B+42 \right) {s}^{4}- \left( 20\,B+7 \right) {s}^{3}-\left( 28\,B+20+36\,{B}^{2} \right) {s}^{2}}  {+\left( 9+10\,B+8\,{B}^{2} \right) s-1+8\,{B}^{3}+8\,{B}^{2}\bigr)}  } { \left(1-s \right) ^{3} \left( 1-3\,{s}^{2}+2\,B \right) ^{3}}, \\
\zeta_{10}=& \dfrac {\splitdfrac{48\bigl( 45\,{s}^{4}-60\,{s}^{3}+ \left( 18-12\,B \right) {s}^{2}+ \left( -4-8\,B \right) s} {+5+4\,{B}^{2}+12\,B \bigr) \left( {B}^{2}-2\,{s}^{2}B+{s}^{4} \right) }}{ \left( 1-3\,{s}^{2}+2\,B \right)^{4} \left( 1-s \right) ^{4}}\\-&
\dfrac {32 \big( 60\,{s}^{4}-72\,{s}^{3}+ \left( 9-21\,B \right) {s}^{2}+ \left( 4-4\,B \right) s+5+13\,B+6\,{B}^{2} \big)  \left( B-{s}^{2} \right) }{ \left( 1-s \right) ^{4} \left( 1-3\,{s}^{2}+2\,B\right) ^{3}}\\&+
\dfrac {16\Big(-17\,{s}^{3}+2\,{s}^{2}+ \left( 4\,B+8 \right) s+1+2\,B\Big)} {\left( 1-s \right) ^{3} \left( 1-3\,{s}^{2}+2\,B \right) ^{2}}\\&+
\dfrac {4\Big(55\,{s}^{4}-60\,{s}^{3}+ \left( 10-52\,B \right) {s}^{2}+
	\left( 32\,B+12 \right) s-5-4\,B+12\,{B}^{2}\Big)}{ \left( 1-s \right) ^{4} \left( 1-3\,{s}^{2}+2\,B \right) ^{2}}, \\
\zeta_{11}=& \dfrac {-4}{ \left( 1-3\,{s}^{2}+2\,B \right) ^{2}}, \\
\zeta_{12}=& \dfrac {-32\,\Big(1-2\,s\Big)}{ \left( 1-s\right) \left( 1-3\,{s}^{2}+2\,B \right) ^{2}},\\
\zeta_{13}=& \dfrac {8\Big(21\,{s}^{3}-3\,{s}^{2}- \left( 26\,B-5 \right) s+2\,B+1\Big)}{\left( 1-s \right)  \left( 1-3\,{s}^{2}+2\,B \right) ^{3}},
\hspace{5.5cm} \\
\zeta_{14}=&   \dfrac {2}{  1-3\,{s}^{2}+2\,B }, \\
\zeta_{15}=&  \dfrac {8\,\left( 1-2\,s\right) }{ \left( 1-s\right)  \left( 1-3\,{s}^{2}+2\,B \right)},\\
\zeta_{16}=& \dfrac {2\Big(21\,{s}^{4}-18\,{s}^{3}+ \left( -2-22\,B \right) {s}^{2}+ \left( -2+20\,B \right) s+1+2\,B\Big)}{ \left( 1-s \right)^2  \left( 1-3\,{s}^{2}+2\,B \right) ^{2}}, \\
\zeta_{17}=&   \dfrac {4\,\left( 1-2\,s\right) }{ \left( 1-s\right)  \left( 1-3\,{s}^{2}+2\,B \right)}, \\ 
\zeta_{18}=
&\dfrac {2\, \Big(15\,{s}^{4}-18\,{s}^{3}+ \left( 2-14\,B \right) {s}^{2}+ \left( 2+16\,B \right) s-2\,B-1\Big)}{\left( 1-s \right)^2  \left( 1-3\,{s}^{2}+2\,B \right) ^{2}}, \\
\zeta_{19}=& \dfrac {-4\,\Big(1-3\,s\Big)}{ \left( 1-s \right) ^{3}}, 
\end{align*}
\begin{equation}\label{smzeta}
\begin{split}
\zeta_{20}=& \dfrac {16}{\left( 1-s\right) \left( 1-3\,{s}^{2}+2\,B \right) ^{2}  },\\
\zeta_{21}=& \dfrac {-16\Big(1-2\,s\Big)}{\left( 1-s\right) ^{2} \left( 1-3\,{s}^{2}+2\,B \right)},  \\
\zeta_{22}=&  \dfrac {4\, \Big(27\,{s}^{3}-3\,{s}^{2}- \left( 26\,B+1 \right) s+1+2\,B\Big)}{\left( 1-s \right) ^{2} \left( 1-3\,{s}^{2}+2\,B \right) ^{2}},  \\
\zeta_{23}=&  \dfrac {4}{ \left( 1-s \right)  \left( 1-3\,{s}^{2}+2\,B \right) }, \\
\zeta_{24}=& \dfrac{4}{1-s}, \\
\zeta_{25}=&  \dfrac {-16}{ \left( 1-s \right)^2  \left( 1-3\,{s}^{2}+2\,B \right) },  \\
\zeta_{26}=&  \dfrac {-4}{ \left( 1-s \right)^2 }.
\end{split}
\end{equation}
From above calculations, we have the following:
\begin{theorem}{\label{RCSM}}
	Let $G/H$ be a homogeneous Finsler space with square metric $F.$
	Then, 	the Ricci curvature  is given by
	\begin{align*}
	Ric(Z)=&^\alpha Ric(Z)+ \dfrac{c^2\left( C^0_{0n}\right)^2 }{\alpha^2(Z)}\Big\{ \left( n-1\right)\zeta_1+\zeta_2  \Big\}
	 +\dfrac{c^3C^0_{0n}C^n_{n0}}{2\alpha(Z)}\Big\{ \left( n-1\right) \left(\zeta_3-\zeta_5 \right) +\zeta_4-\zeta_6 \Big\}\\
	& -\dfrac{cC^0_{0l}\left( C^0_{nl}+C^l_{n0}\right) }{\alpha(Z)}\Big\{ \left( n-1\right)\zeta_7+\zeta_8  \Big\} 
	  + \dfrac{c^4\left( C^n_{n0}\right)^2 }{4}\Big\{ \left( n-1\right)\left( \zeta_9-\zeta_{12} \right) +\zeta_{10}-\zeta_{11}-\zeta_{13} \Big\}\\
	&  - \dfrac{c^2}{4}\Big\{ 4C^0_{0n}C^l_{nl }+\left(C^l_{n0}+C^0_{nl} \right) \left( 2C^0_{nl}+C^n_{0l}+2C^l_{0n}\right) + 2 C^n_{nl}C^0_{l0}\Big\} \zeta_{14} \\
	& + \dfrac{c^2}{4}C^n_{0l}\left( C^l_{n0}+C^0_{nl}\right)\Big\{ \left( n-1\right)\zeta_{15}+\zeta_{16}\Big\} 
	 + \dfrac{c^2}{2}C^n_{nl}C^0_{0l} \Big\{ \left( n-1\right)\zeta_{17}+\zeta_{18}\Big\} \\
	&  - \dfrac{c^2}{4}\left( C^n_{l0}\right)^2\zeta_{19} + \dfrac{c^3}{4}\alpha(Z)C^n_{l0}C^n_{nl}\Big\{ \left( n-1\right)\zeta_{21}+\zeta_{22}\Big\} \\
	&  + \dfrac{c^3}{4}\alpha(Z)\Big\{  4C^n_{n0}C^l_{nl}-C^n_{nl}\left( 4C^0_{nl}-C^n_{0l}\right)   \Big\} \zeta_{23} 
	   +\dfrac{c}{4}\alpha(Z)\Big\{ 2C^n_{t 0}C^l_{l t}+C^n_{l t}C^0_{l t}\Big\} \zeta_{24} \\&
	+ \dfrac{c^2}{4}\alpha^2(Z) \Big\{c^2 \left( C^n_{nl}\right)^2 \zeta_{25}-\left( C^n_{l i}\right)^2 \zeta_{26}  \Big\}, 
	\end{align*}
	where $Z$ is a non-zero vector in $   \mathfrak{k},$ and  $\zeta_1$ to $\zeta_{26}$  are given by (\ref{smzeta}).
\end{theorem} 
Next,  $\zeta_1$ to $\zeta_{26}$ for   Randers change of  square metric are calculated as follows:
\begin{align*}
\zeta_1
=& \dfrac {\splitdfrac{ 3\,\bigl(-40\,{s}^{8}-192\,{s}^{7}+ \left( 16\,B-235 \right) {s}^{6} +\left( 72\,B+96 \right) {s}^{5}+ \left( 235+90\,B \right) {s}^{4}}{-\left( 73+20\,B \right){s}^{2} +24\,sB+18\,B+9\bigr)}}{4\, \left( 1-3\,{s}^{2}+2\,B \right) ^{3} \left( 1+3\,s+{s}^{2} \right) ^{2}},\\
\zeta_2=& \dfrac {-6\,\Big(9\,{s}^{6}+ \left( -13\,B-2 \right) {s}^{4}+ \left( 2\,{B}^{2}+4\,B-3 \right) {s}^{2}+B+2\,{B}^{2}\Big)} { \left(1-3\,{s}^{2}+2\,B\right) ^{4}}, 
\end{align*}
\begin{align*}
\zeta_3=&  \dfrac { \splitdfrac{  4\,  \left( 2\,s+3 \right)  \bigl( 12\,{s}^{6}+54\,{s}^
		{5}+ \left( 73+8\,B \right) {s}^{4}+ \left( 48\,B-12 \right) {s}^{3}   }{+
		\left( -42+78\,B \right) {s}^{2}+ \left( 6+48\,B \right) s+18\,B+9
		\bigr) \left( B-{s}^{2} \right)} }{ \left( 1-{s}^{2} \right)  \left( 1+3\,s+{s}^{2} \right) ^{
		2} \left( 1-3\,{s}^{2}+2\,B \right) ^{3}}\\
&-{\dfrac { 4\,(3-9\,{s}^{2}-4\,{s}
		^{3})\left( B-{s}^{2} \right)}{ \left( 1-{s}^{2} \right) ^{2} \left( 1-3\,{s}^{2}+2\,B \right) 
		^{2}}} 
-	{\dfrac {2\, \left( 2\,s+3 \right)  \left( -3+9\,{s}^{2}+4\,{s}^{3}
		\right) s}{ \left( 1-{s}^{2} \right)  \left( 1+3\,s+{s}^{2} \right) 
		\left( 1-3\,{s}^{2}+2\,B \right) ^{2}}}\\
&
-
\dfrac {  \splitdfrac{   \left( 2\,s+3 \right)  \bigl( 40\,{s}^{6}+180\,{s}^{5}+
		\left( 227+16\,B \right) {s}^{4}+ \left( 96\,B-48 \right) {s}^{3}} {-
		\left( 138-156\,B \right) {s}^{2}+ \left( 12+96\,B \right) s+36\,B+
		27 \bigr)} }{ \left( 1+3\,s+{s}^{2} \right) ^{2} \left( 1-3\,{s}^{2}+2
	\,B \right) ^{2} \left( 1-{s}^{2} \right) }\\
&-{\dfrac {2\, \left( -3+9\,{s}^{2}+4\,{s}^{3} \right)  \left( 2-4\,{s}^
		{2}+2\,B+3\,s-9\,{s}^{3}+6\,sB-2\,{s}^{4}+2\,{s}^{2}B \right) }{
		\left( 1-{s}^{2} \right) ^{2} \left( 1+3\,s+{s}^{2} \right)  \left( 
		1-3\,{s}^{2}+2\,B \right) ^{2}}},  \\
\zeta_4=& - 24\, \biggl( {\dfrac { \left( 3+2\,s \right)  \left( 18\,{s}^{2}+2+4\,B+
		s-3\,{s}^{3}+2\,Bs \right) }{ \left( 1-3\,{s}^{2}+2\,B \right) ^{4}
		\left( 1-{s}^{2} \right) }} \\& \ \  +{\dfrac {12\, \left( 6\,s+2\,{s}^{3} +9\,{s}^{2}
		+3 \right) {s}^{2}}{ \left( 1-{s}^{2} \right) ^{3} \left( 1-3\,{s}^{2}+2\,B \right) ^{5}}} \biggr)  \left( B-{s}^{2} \right) ^{2}\\&+
4\, \biggl( {\dfrac {10\,{s}^{3}-3+27\,{s}^{2}+6\,s}{ \left( 1-{s}^{2}
		\right) ^{2} \left( 1-3\,{s}^{2}+2\,B \right) ^{2}}}+{\dfrac {2\,\left( 6\,s+
		2\,{s}^{3}+9\,{s}^{2}+3\right) }{ \left( 1-{s}^{2} \right) ^{3} \left( 1-3\,{s
		}^{2}+2\,B \right) }} \\& \ \ +{\dfrac {6\, \left( 2\,s+3 \right)  \left( 14\,{s
		}^{2}+1+2\,B \right) }{ \left( 1-3\,{s}^{2}+2\,B \right) ^{3} \left( 1
		-{s}^{2} \right) }} +{\dfrac {12\,s}{ \left( 1-3\,{s}^{2}+2\,B \right) ^
		{3}}} \biggr)  \left( B-{s}^{2} \right) \\&
-{\dfrac {2\,\left( 4\,{s}^{3}-3+9\,{s}^{2}\right) }{ \left( 1-3\,{s}^{2}+2\,B\right)  \left( 1-{s}^{2} \right) ^{2}}}
-{\dfrac {12\,s \left( 2+8\,{s}^{2}+15\,s \right) }{ \left( 1-3\,{s}^{2}+2\,B \right) ^{2} \left( 1-{s}^{2} \right) }}\\&
-{\dfrac {2\,\left( 6\,s+2\,{s}^{3}+9\,{s}^{2}+3\right) }{ \left( 
		1-{s}^{2} \right) ^{3}}},\\
\zeta_5=& \dfrac {-4\,\left( -3+9\,{s}^{2}+4\,{s}^{3}\right) }{ \left( 1+3\,s+{s}^{2} \right) 
	\left( 1-3\,{s}^{2}+2\,B \right) ^{2}},  \\
\zeta_6=& {\dfrac {12\,s \left( 4\,B-3\,{s}^{2}-1 \right) }{ \left( 1-3\,{s}^{2}+
		2\,B \right) ^{3}}}, \\
\zeta_7=& {\dfrac {-3+9\,{s}^{2}+4\,{s}^{3}}{2\, \left( 1+3\,s+{s}^{2} \right) 
		\left( 1-3\,{s}^{2}+2\,B \right) }}, \\
\zeta_8=& {\dfrac {-6\,s\Big(B-s^2\Big)}{ \left( 1-3\,{s}^{2}+2\,B \right) ^{2}}},\\
\zeta_9=
& \dfrac { 8\,\left( 3+2\,s \right)  \left( -3+9\,{s}^{2}+4\,{s}
	^{3} \right)\left( B-{s}^{2} \right) }{ \left( 1-{s}^{2} \right) ^{3} \left( 1-3\,{s}^{2}+2\,B
	\right) ^{2}}\\
&-\dfrac {\splitdfrac{4\, \left( 3+2\,s \right) ^{2} \bigl(12\,{s}^{6}+54\,{s}^{5}+ \left( 73+8\,B \right) {s}^{4}+ \left( 48\,B-
		12 \right) {s}^{3}}{+ \left( -42+78\,B \right) {s}^{2}+ \left( 6+48\,B
		\right) s+18\,B+9
		\bigr)\left( B-{s}^{2} \right)} }{ \left( 1-{s}^{2}
	\right) ^{2} \left( 1+3\,s+{s}^{2} \right) ^{2} \left( 1-3\,{s}^{2}+2
	\,B \right) ^{3}}  \\
&+{\dfrac { \left( 3+2\,s \right) ^{2} \left( 3-9\,{s}^{2}-4\,{s}^{3}
		\right) ^{2}}{ \left( 1-{s}^{2} \right) ^{2} \left( 1+3\,s+{s}^{2}
		\right) ^{2} \left( 1-3\,{s}^{2}+2\,B \right) ^{2}}}\\
 &+{\dfrac {  \splitdfrac{  2\,
			\left( 3+2\,s \right) ^{2} \bigl( 12\,{s}^{6}+54\,{s}^{5}+ \left( 73+
			8\,B \right) {s}^{4}+ \left( 48\,B-12 \right) {s}^{3}}{+ \left( -42+78\,
			B \right) {s}^{2}+ \left( 6+48\,B \right) s+18\,B+9 \bigr)} }{ \left( 
		1-{s}^{2} \right) ^{2} \left( 1+3\,s+{s}^{2} \right) ^{2} \left( 1-3\,
		{s}^{2}+2\,B \right) ^{2}}}\\
&-{\dfrac {4\, \left( 3+2\,s \right)  \left( 3-9\,{s}^{2}-4\,{s}^{3}
		\right) }{ \left( 1-{s}^{2} \right)  \left( 1+3\,s+{s}^{2} \right) 
		\left( 1-3\,{s}^{2}+2\,B \right) ^{2}}}
-{\dfrac {8\, \left( 3+2\,s
		\right)  \left( 1+3\,s+{s}^{2} \right) }{ \left( 1-{s}^{2} \right) ^{3}}},
	\end{align*}
	\begin{align*}
\zeta_{10}=
& \Biggl\{ \dfrac {16\,\Big(18\,s+28\,{s}^{2}+18\,{s}^{3}+3\,{s}^{4}+8\Big)}{ \left( 1-{s}^{2} \right) ^{4} \left( 1-3\,{s}^{2}+2\,B \right) ^{2}}
-\dfrac {144\,{s}^{2} \left( 3+2\,s \right) ^{2}}{ \left( 1-3\,{s}^{2}+2\,B \right) ^{4} \left( 1-{s}^{2} \right) ^{2}}\\&
+\dfrac {\splitfrac{ 48\,\left( 3+2\,s \right)  \bigl( -24\,{s}^{5}-45\,{s}^{4}+
		12\,{s}^{3}} {+ \left( 6\,B+30 \right) {s}^{2}+ \left( 4+8\,B \right) s+3
		+6\,B \bigr) }}{ \left( 1-3\,{s}^{2}+2\,B \right) ^{4} \left( 1-{s}^{
		2} \right) ^{3}}
\Biggr\}(B-s^2)^2\\&-
\Biggl\{ \dfrac {\splitfrac{8\, \left( 3+2\,s \right)  \bigl( -108\,{s}^{5}-207\,{s}^{4}+
		\left( 8\,B+52 \right) {s}^{3}} {+ \left( 42\,B+138 \right) {s}^{2}+
		\left( 16+32\,B \right) s+9+18\,B \bigr) }} { \left( 1-{s}^{2}
	\right) ^{3} \left( 1-3\,{s}^{2}+2\,B \right) ^{3}}\\& 
+{\dfrac { 16\,\Big(18\,s+28\,{s}^{2}+18\,{s}^{3}+3\,{s}^{4}+8\Big)}{ \left( 1-{s}^{2} \right) ^{4} \left( 1-3\,{s}^{2}+2\,B \right) }}\\& 
+{\dfrac {16\, \Big(-15\,{s}^{4}-27\,{s}^{3}+ \left( 2\,B+10 \right) {s}^{2}+
		\left( 6\,B+21 \right) s+1+2\,B\Big)}{ \left( 1-3\,{s}^{2}+2\,B \right) ^{3} \left( 1-{s}^{2} \right) ^{2}}} 
\Biggr\}(B-s^2)\\
&+{\dfrac {8\, \left( 2+4\,{s}^{2}+9\,s \right)  \left( -9\,{s}^{4}-18\,
		{s}^{3}+ \left( 2\,B+4 \right) {s}^{2}+ \left( 6\,B+12 \right) s+1+2\,
		B \right) }{ \left( 1-{s}^{2} \right) ^{3} \left( 1-3\,{s}^{2}+2\,B
		\right) ^{2}}}\\&
-{\dfrac {4\, \left( 3+2\,s \right)  \left( -12\,{s}^{3}-15\,{s}^{2}+
		\left( 6+8\,B \right) s+6+12\,B \right) }{ \left( 1-3\,{s}^{2}+2\,B
		\right) ^{2} \left( 1-{s}^{2} \right) ^{2}}}
+\\&
{\dfrac {418\,s+28\,{s}^{2}+18\,{s}^{3}+3\,{s}^{4}+8}{ \left( 1-{s}^
		{2} \right) ^{4}}},\\
\zeta_{11}=& \dfrac {-4}{ \left( 1-3\,{s}^{2}+2\,B \right) ^{2}}, \\
\zeta_{12}=& {\dfrac {8\, \left( 3+2\,s \right)  \left( -3+9\,{s}^{2}+4\,{s}^{3}
		\right) }{ \left( 1+3\,s+{s}^{2} \right)  \left( 1-3\,{s}^{2}+2\,B
		\right) ^{2} \left( 1-{s}^{2} \right) }}, \\
\zeta_{13}=&  {\dfrac {8\,\Big(21\,{s}^{4}+27\,{s}^{3}- \left( 26\,B-2 \right) {s}^{2}-\left( 36\,B-9 \right) s+1+2\,B\Big)}{ \left( 1-3\,{s}^{2}+2\,B \right) ^{3} \left( 1-{s}^{2} \right) }}, \\
\zeta_{14}=&   \dfrac {2}{  1-3\,{s}^{2}+2\,B },  \\
\zeta_{15}=& {\dfrac {2\, \left( 2\,s+3 \right)  \left( 3-9\,{s}^{2}-4\,{s}^{3}
		\right) }{ \left( 1-{s}^{2} \right)  \left( 1+3\,s+{s}^{2} \right) 
		\left( 1-3\,{s}^{2}+2\,B \right) }},\\
\zeta_{16}=& {\dfrac {2\,\Big(-21\,{s}^{4}-36\,{s}^{3}+ \left( 22\,B-4 \right) {s}^{2}+
		36\,sB+2\,B+1\Big)}{ \left( 1-3\,{s}^{2}+2\,B \right) ^{2} \left( 1-{s}^{2} \right) }}, \\
\zeta_{17}=&  {\dfrac { \left( 2\,s+3 \right)  \left( 3-9\,{s}^{2}-4\,{s}^{3}
		\right) }{ \left( 1-{s}^{2} \right)  \left( 1+3\,s+{s}^{2} \right) 
		\left( 1-3\,{s}^{2}+2\,B \right) }},  \\ 
\zeta_{18}=&{\dfrac {-2\,\Big(15\,{s}^{4}+18\,{s}^{3}- \left( 14\,B+4 \right) {s}^{2} -18\,Bs+1+2\,B\Big)}{ \left( 1-{s}^{2} \right)  \left( 1-3\,{s}^{2}+2\,B \right) ^{2}}},\\
\zeta_{19}=& {\dfrac {2\,\Big(-7+27\,{s}^{2}+24\,{s}^{3}+6\,{s}^{4}\Big)}{ \left( 1-{s}^{2}\right) ^{3}}},  \\  
\zeta_{20}=& {\dfrac {8\,\left( 2\,s+3\right) }{ \left( 1-3\,{s}^{2}+2\,B \right) ^{2} \left( 1-{s}^{2} \right) }}, \\
\zeta_{21}=& \dfrac {2\, \left( 2\,s+3 \right) ^{2} \left( -3+9\,{s}^{2}+4\,{s}^{3}
	\right) }{ \left( 1-{s}^{2} \right) ^{2} \left( 1+3\,s+{s}^{2}
	\right)  \left( 1-3\,{s}^{2}+2\,B \right) }, 
\end{align*}

\begin{equation}\label{rcsmzeta}
\begin{split}
\zeta_{22}=&  {\dfrac {-2\, \left( 2\,s+3 \right)  \left( -27\,{s}^{4}-36\,{s}^{3}+
		\left( 26\,B+4 \right) {s}^{2}+36\,Bs-1-2\,B \right) }{ \left( 1-3\,{
			s}^{2}+2\,B \right) ^{2} \left( 1-{s}^{2} \right) ^{2}}},  \hspace*{2cm}\\
\zeta_{23}=& {\dfrac {2\,\left( 2\,s+3\right) }{ \left( 1-{s}^{2} \right)  \left( 1-3\,{s}^{2}+2\,B \right) }},  \\
\zeta_{24}=&  {\dfrac {2\,\left( 2\,s+3\right) }{1-{s}^{2}}}, \\
\zeta_{25}=& {\dfrac {-4\,\left( 9+4\,{s}^{2}+12\,s\right) }{ \left( 1-{s}^{2} \right) ^{2} \left( 1-3\,{s}^{2}+2\,B \right) }},  \\
\zeta_{26}=&   {\dfrac {-4\,\left( 9+4\,{s}^{2}+12\,s\right) }{ \left( 1-{s}^{2} \right) ^{2}}}.
\end{split}
\end{equation}
\begin{theorem}{\label{RCRCSM}}
	Let $F$
	be Randers change of  square metric on  a homogeneous Finsler space $G/H.$ Then, 	the Ricci curvature  is given by
	\begin{align*}
	Ric(Z)=&^\alpha Ric(Z)+ \dfrac{c^2\left( C^0_{0n}\right)^2 }{\alpha^2(Z)}\Big\{ \left( n-1\right)\zeta_1+\zeta_2  \Big\}
	 +\dfrac{c^3C^0_{0n}C^n_{n0}}{2\alpha(Z)}\Big\{ \left( n-1\right) \left(\zeta_3-\zeta_5 \right) +\zeta_4-\zeta_6 \Big\}\\
	& -\dfrac{cC^0_{0l}\left( C^0_{nl}+C^l_{n0}\right) }{\alpha(Z)}\Big\{ \left( n-1\right)\zeta_7+\zeta_8  \Big\} 
	  + \dfrac{c^4\left( C^n_{n0}\right)^2 }{4}\Big\{ \left( n-1\right)\left( \zeta_9-\zeta_{12} \right) +\zeta_{10}-\zeta_{11}-\zeta_{13} \Big\}\\
	&  - \dfrac{c^2}{4}\Big\{ 4C^0_{0n}C^l_{nl }+\left(C^l_{n0}+C^0_{nl} \right) \left( 2C^0_{nl}+C^n_{0l}+2C^l_{0n}\right) + 2 C^n_{nl}C^0_{l0}\Big\} \zeta_{14} \\
	& + \dfrac{c^2}{4}C^n_{0l}\left( C^l_{n0}+C^0_{nl}\right)\Big\{ \left( n-1\right)\zeta_{15}+\zeta_{16}\Big\} 
	 + \dfrac{c^2}{2}C^n_{nl}C^0_{0l} \Big\{ \left( n-1\right)\zeta_{17}+\zeta_{18}\Big\} \\
	&  - \dfrac{c^2}{4}\left( C^n_{l0}\right)^2\zeta_{19} + \dfrac{c^3}{4}\alpha(Z)C^n_{l0}C^n_{nl}\Big\{ \left( n-1\right)\zeta_{21}+\zeta_{22}\Big\} \\
	 & + \dfrac{c^3}{4}\alpha(Z)\Big\{  4C^n_{n0}C^l_{nl}-C^n_{nl}\left( 4C^0_{nl}-C^n_{0l}\right)   \Big\} \zeta_{23} \\
	&   +\dfrac{c}{4}\alpha(Z)\Big\{ 2C^n_{t 0}C^l_{l t}+C^n_{l t}C^0_{l t}\Big\} \zeta_{24} 
	+ \dfrac{c^2}{4}\alpha^2(Z) \Big\{c^2 \left( C^n_{nl}\right)^2 \zeta_{25}-\left( C^n_{l i}\right)^2 \zeta_{26}  \Big\}, 
	\end{align*}
	where $Z$ is a non-zero vector in $   \mathfrak{k},$ and  $\zeta_1$ to $\zeta_{26}$  are given by (\ref{rcsmzeta}).
\end{theorem}  
\section{Ricci Curvature  of  Homogeneous Finsler Spaces with Certain $(\alpha, \beta)$-Metrics Having Vanishing $S$-Curvature}
We have already given the formula for $S$-curvature of  homogeneous Finsler space with $G$-invariant square  metric in \cite{GSSR2019},  and with   $G$-invariant Randers change of square metric in \cite{SRGS2019}. Here, we give the equivalent  condition for these spaces to have vanishing $S$-curvature.
\begin{theorem}{\label{vanish_s_cur}}
	Let $G/H$ be a compact homogeneous Finsler space and $F$ be a  $G$-invariant  square  metric on $G/H.$  Then $\left( G/H,F\right) $ has vanishing $S$-curvature if and only if $ \left\langle  \left[ w,Y\right]_{\mathfrak k} , Y\right\rangle =0 \ \forall \  Y \in \mathfrak{k}$ and $w \in \mathfrak{k}$ corresponds to the 1-form $\beta.$
\end{theorem}
\begin{proof}
	We know that $F$ has vanishing $S$-curvature if and only if 
	\begin{equation}\label{ss.1}
	r_{ij}=0,   \text{ and \ } s_i=0 \ \ \forall\ 1 \leq i,\ j \leq n.
	\end{equation}
	Therefore, we have to prove that the conditions in equation (\ref{ss.1}) are  equivalent to $ \left\langle  \left[ w,Y\right]_{\mathfrak k} , Y\right\rangle =0 \ \forall \  Y \in \mathfrak{k}.$\\
	Firstly, suppose that $  r_{ij}=0,$   and  $s_i=0 \ \ \forall\ 1 \leq i,\ j \leq n. $
	Using Lemma \ref{lemma3.2}, we have 
	\begin{equation}\label{ss.2}
	\dfrac{c}{2}\left(C^j_{in} + C^i_{jn}\right)=0,
	\end{equation}
	and 
	\begin{equation}\label{ss.3}
	\dfrac{c^2}{2}C^n_{ni}=0.
	\end{equation}
	For $i=j,$ equation  (\ref{ss.2}) gives us 
	$
	\left\langle \left[ w,w_i\right]_{\mathfrak k}, w_i \right\rangle=0.$\\
	Since $\left\{ w_1, w_2, w_3, ... , w_n=\dfrac{w}{c}  \right\}$ is  the orthonormal basis of $\mathfrak{k},$  we have 
	$$ \left\langle  \left[ w,Y\right]_{\mathfrak k} , Y\right\rangle =0 \ \forall \  Y \in \mathfrak{k}.$$
	For the second part, suppose $$ \left\langle  \left[ w,Y\right]_{\mathfrak k} , Y\right\rangle =0 \ \forall \  Y \in \mathfrak{k}.$$
	Therefore, 
	\begin{equation}\label{ss.4}
	\left\langle \left[ w,w_i\right]_{\mathfrak k}, w_i \right\rangle=0 \ \forall \ 1 \leq i \leq n,
	\end{equation}
	\begin{equation}\label{ss.5}
	\left\langle \left[ w,w_i+w_j\right]_{\mathfrak k}, w_i+w_j \right\rangle=0 \ \forall \ 1 \leq i,\ j \leq n.
	\end{equation}
	\begin{equation}\label{ss.6}
	\left\langle \left[ w,w+w_i\right]_{\mathfrak k}, w+w_i \right\rangle=0 \ \forall \ 1 \leq i \leq n.
	\end{equation}
	From (\ref{ss.4}) and (\ref{ss.5}), we get $C^i_{nj} + C^j_{ni}=0,$ \ i.e., \ $ r_{ij}=0. $\\
	Further, from (\ref{ss.4}) and (\ref{ss.6}), we get $\left\langle \left[ w,w_i\right]_{\mathfrak k}, w \right\rangle=0,$ \ i.e., \ $ s_{i}=0. $
\end{proof}
Similarly, we can prove following theorem: 
\begin{theorem}{\label{vanish_RCs_cur}}
	Let $G/H$ be a compact homogeneous Finsler space and $F$ be a  $G$-invariant  Randers change of square  metric on $G/H.$  Then $\left( G/H,F\right) $ has vanishing $S$-curvature if and only if $ \left\langle  \left[ w,Y\right]_{\mathfrak k} , Y\right\rangle =0 \ \forall \  Y \in \mathfrak{k}$ and $w \in \mathfrak{k}$ corresponds to the 1-form $\beta.$
\end{theorem}

\begin{theorem}
	Let $G/H$ be a compact homogeneous Finsler space with $G$-invariant square  metric $F.$ If $\left( G/H,F\right) $ has vanishing $S$-curvature, then Ricci curvature is given by 
	\begin{equation*}
	Ric(Z)= Ric^{\alpha}(Z)-\dfrac{c^2}{4}\left( C^n_{l 0}\right)^2 \zeta_{19} + \dfrac{c}{4}\alpha(Z)\left( 2\, C^n_{t 0}C^{l }_{l t}+C^n_{l  t}C^0_{l t}\right) \zeta_{24} - \dfrac{c^2}{4} \alpha^2(Z)\left( C^n_{ik}\right)^2 \zeta_{26},  
	\end{equation*}
	where $Z$ is a non-zero vector in $   \mathfrak{k},$ and  $\zeta_{19}, \ \zeta_{24},\ \zeta_{26}  $   are same as  given in (\ref{smzeta}). 
\end{theorem}
\begin{proof}
	Since  $\left( G/H,F\right) $ has vanishing $S$-curvature, by  Theorem \ref{vanish_s_cur},
	$$ \left\langle  \left[ w,Y\right]_{\mathfrak k}, Y  \right\rangle =0 \ \forall\ Y \in \mathfrak{k}.$$
	Therefore, 
	\begin{align*}
	C^n_{nl}&=\left\langle \left[ w_n,w_l\right]_{\mathfrak k}, w_n \right\rangle \\
	&=\dfrac{1}{c^2}\big\{\left\langle \left[ w, w_l\right]_{\mathfrak k}, w \right\rangle + \left\langle \left[ w,w_l\right]_{\mathfrak k}, w_l \right\rangle\big\}\\
	&=\dfrac{1}{c^2}\left\langle \left[ w, w_l\right]_{\mathfrak k}, w+w_l \right\rangle \\
	&=\dfrac{1}{c^2}\left\langle \left[ w,w\right]_{\mathfrak k} + \left[ w, w_l\right]_{\mathfrak k}, w+w_l \right\rangle\\
	&=\dfrac{1}{c^2}\left\langle \left[ w,w+w_l\right]_{\mathfrak k}, w+w_l \right\rangle\\
	&=0 \ \ \ \ \ \forall\  1 \leq l \leq n.
	\end{align*}
	Further,  for non-zero $X \in \mathfrak k,$ we have 
	\begin{align*}
	C^n_{n0}&=\left\langle \left[ w_n,X\right]_{\mathfrak k}, w_n \right\rangle \\
	&=\dfrac{1}{c^2}\big\{ \left\langle \left[ w, X\right]_{\mathfrak k}, w \right\rangle + \left\langle \left[ w,X\right]_{\mathfrak k}, X \right\rangle\big\}\\
	&=\dfrac{1}{c^2}\left\langle \left[ w, X\right]_{\mathfrak k}, w+X \right\rangle \\
	&=\dfrac{1}{c^2}\left\langle \left[ w,w+X\right]_{\mathfrak k}, w+X \right\rangle
	=0,\\
	C^0_{0n} &=\left\langle \left[X, w_n\right]_{\mathfrak k}, X  \right\rangle
	=\left\langle  \left[X, \dfrac{w}{c}\right]_{\mathfrak k}, X  \right\rangle=0,
	\end{align*}
	and
	\begin{align*}
	C^0_{nl}+C^l_{n0}&= \left\langle \left[ w_n,w_l\right]_{\mathfrak k}, X \right\rangle + \left\langle \left[ w_n,X\right]_{\mathfrak k},  w_l \right\rangle\\
	&=\dfrac{1}{c}\Big\{ \left\langle\left[ {w},w_l\right]_{\mathfrak k},  X  \right\rangle + \left\langle  \left[ w, X\right]_{\mathfrak k}, w_l  \right\rangle \Big\}\\
	&=\dfrac{1}{c}\Big\{ \left\langle \left[ w,w_l\right]_{\mathfrak k},  X  \right\rangle
	+ \left\langle  \left[ w,X\right]_{\mathfrak k}, w_l  \right\rangle 
	+ \left\langle  \left[ w,w_l\right]_{\mathfrak k}, w_l  \right\rangle 
	+ \left\langle  \left[ w,X\right]_{\mathfrak k}, X  \right\rangle\Big\}\\
	&=\dfrac{1}{c}\Big\{ \left\langle \left[ w,w_l\right]_{\mathfrak k}, X+w_l  \right\rangle + \left\langle \left[ w,X\right]_{\mathfrak k}, X+w_l  \right\rangle \Big\}\\
	&= \dfrac{1}{c}\Big\{ \left\langle \left[ w,w_l\right]_{\mathfrak k} + \left[ w,X\right]_{\mathfrak k}, X+w_l \right\rangle  \Big\}\\
	&=\dfrac{1}{c}\left\langle  \left[ w,X+w_l\right]_{\mathfrak k}, X+w_l  \right\rangle \\
	&=0.
	\end{align*}
	Using  the above values in Theorem \ref{RCSM}, we get the required result.
\end{proof}
In similar manner, we can prove the following theorem:
\begin{theorem}
	Let $G/H$ be a compact homogeneous Finsler space and $F$ be a  $G$-invariant Randers change of square  metric on $G/H.$ If $\left( G/H,F\right) $ has vanishing $S$-curvature, then Ricci curvature is given by 
	\begin{equation*}
	Ric(Z)= Ric^{\alpha}(Z)-\dfrac{c^2}{4}\left( C^n_{l 0}\right)^2 \zeta_{19} + \dfrac{c}{4}\alpha(Z)\left( 2\, C^n_{t 0}C^{l }_{l t}+C^n_{l  t}C^0_{l t}\right) \zeta_{24} - \dfrac{c^2}{4} \alpha^2(Z)\left( C^n_{ik}\right)^2 \zeta_{26},  
	\end{equation*}
	where $Z$ is a non-zero vector in $   \mathfrak{k},$ and  $\zeta_{19}, \ \zeta_{24},\ \zeta_{26}  $   are same as  given in (\ref{RCRCSM}). 
\end{theorem}
\begin{cor}
	Let $ G/H$ be a homogeneous Finsler space and $F$ be a  $G$-invariant square  metric or Randers change of square  metric on $G/H.$ If $\left( G/H,F\right) $ has vanishing $S$-curvature  and negative Ricci curvature, then it must be Riemannian.
\end{cor}

\section*{Acknowledgements}
	First author is very much thankful to UGC for providing financial assistance in terms of JRF fellowship vide letter with  Sr. No. 2061641032 and  Ref. No. 19/06/2016(i)EU-V.

\end{document}